\documentclass[11pt,leqno]{amsproc}
\usepackage[margin=1in]{geometry}
\usepackage{amsmath, amsthm, amssymb}
\usepackage[colorlinks=true, pdfstartview=FitV, linkcolor=blue,citecolor=blue, urlcolor=blue]{hyperref}
\usepackage[abbrev,lite,nobysame]{amsrefs}
\usepackage{times}
\usepackage[usenames,dvipsnames]{color}
\usepackage[T1]{fontenc}
\usepackage[utf8]{inputenc} 
\usepackage{todonotes}
\usepackage{mathrsfs}

\usepackage{dsfont}

\def\R {\mathbb{R}}

\def\d{\diamond}

\def\R{\mathbb{R}}

\def\C{\mathbb{C}}

\newcommand{\wh}{\widehat}

\def\pr{\right )}
\def\le{\left (}

\def\d{\,d}

\def\f{\varphi}

\def\e{\varepsilon}
\def\1{\mathds{1}} 

\newtheorem{proposition}{Proposition}[section]
\newtheorem{theorem}[proposition]{Theorem}
\newtheorem{corollary}[proposition]{Corollary}
\newtheorem{lemma}[proposition]{Lemma}
\theoremstyle{definition}

\newtheorem{remark}[proposition]{Remark}

\numberwithin{equation}{section}

\title[Brezis--Kato Type Results]{Brezis--Kato Type Regularity Results for Higher Order Elliptic Operators}

\author[J. Siemianowski]{Jakub Siemianowski}
\address[J. Siemianowski]{	\newline\indent 
	Institute of Mathematics,
	\newline\indent 
	Polish Academy of Sciences,
	\newline\indent 
	ul. \'Sniadeckich 8, 00-656 Warsaw, Poland
	\newline\indent 
	and
	\newline\indent
	Faculty of Mathematics and Computer Sciences,
	\newline\indent 
	Nicolaus Copernicus University in Toru\'{n} 
	\newline\indent 
	ul. Gagarina 11, 87-100 Toruń, Poland
	\newline\indent 
}
\email{\href{mailto:jsiemianowski@impan.pl}{jsiemianowski@impan.pl}}

\email{\href{mailto:jsiem@mat.umk.pl}{jsiem@mat.umk.pl}}



\begin{document}

\begin{abstract}
We prove a Brezis--Kato regularity type results for solutions of the higher order nonlinear elliptic equation
\[
	L u = g(x,u)\qquad\text{in }\Omega
\]
with an elliptic operator $L$ of $2m$ order with  variable coefficients and a Carath\'eodory function $g:\Omega\times \C\to\C$, where $\Omega\subset\R^N$ is an open set with $N > 2m$. 
\end{abstract}

\maketitle

\section{Notations and Intorduction}
The aim of this paper is to generalize the Brezis--Kato theorem \cite{Brezis_Kato} for higher order elliptic differential operators with variable coefficients.
It is intended to make the least assumptions possible concerning the coefficients of the operator so that it can serve for future references.

We consider a linear differential operator $L = (-1)^m \sum_{|\alpha|\leq 2m}a_\alpha(x)D^\alpha$ with variable complex-valued coefficients defined in an arbitrary open subset $\Omega$ of $\R^N$.
Here for every multi-index $\alpha = (\alpha_1,\ldots, \alpha_N)$ we set $D^\alpha = \frac{\partial^{|\alpha|}}{\partial x_1^{\alpha_1}\ldots\partial x_N^{\alpha_N}}$ and  $\xi^\alpha = \xi_1^{\alpha_1}\ldots\xi_N^{\alpha_N}$, for a vector $\xi = (\xi_1,\ldots, \xi_N)$ .
The characteristic form $l(x,\xi)$ of $L$ is defined by $l(x,\xi) = \sum_{|\alpha|= 2m}a_\alpha(x)\xi^\alpha$ for $x\in \Omega$ and real vectors $\xi$.
We say that $L$ is \emph{strongly elliptic} if  there exists $\lambda>0$ such that 
\begin{equation}\label{elliticity}
\mathfrak{Re} (l(x,\xi)) \geq \lambda |\xi|^{2m} 
\end{equation}
for $x\in \Omega$ and all real vectors $\xi=(\xi_1,\ldots,\xi_N)$.
We make standing assumptions:
\begin{enumerate}
\item[(i)] all coefficients $a_\alpha$ belong to the space $L^\infty_\mathrm{loc}(\Omega)$,
\item[(ii)] top-order coefficients $a_\alpha$ for $|\alpha| = 2m$ are continuous on $\Omega$.
\end{enumerate} 

We denote by $C^k(\Omega)$ (resp. $C^\infty_0(\Omega)$) the class of complex valued functions defined in $\Omega$ which has continuous derivatives up to order $k$ (resp. which are smooth and have compact support).
We consider H\"older's spaces $C^{k,h}(\Omega)$ for nonnegative integers $k$ and $0< h < 1$ which consists of all functions $u\in C^k(\Omega)$ such that $D^\alpha u$ with $|\alpha| = k$ satisfies the (uniform) H\"older condition of exponent $h$.
The symbol $W^{k,p}(\Omega)$ for $1\leq p < \infty$ denotes the Sobolev space of all complex valued $u\in L^p(\Omega)$ for which all the distributional derivatives $D^\alpha u $ are also in $L^p(\Omega)$ for $|\alpha|\leq k$.
We consider the norm $\|u\|_{W^{k,p}(\Omega)} = \le\sum_{|\alpha| \leq k}\|D^\alpha u\|^p_{L^p(\Omega)} \pr^{1/p} $ for $u\in W^{k,p}(\Omega)$.
We use the symbol $W^{k,p}_0(B)$ to denote the closure of $C^\infty_0(\Omega)$ in the $W^{k,p}(\Omega)$ norm.
The local Sobolev space $W^{k,p}_\mathrm{loc}(\Omega)$ consists of functions $u: \Omega\to \C$ such that $u\in W^{k,p}(U)$ for every open subset $U\subset \subset \Omega$ (the latter means that $\overline{U}$ is compact and $\overline{U}\subset \Omega$).
Similarly, $C^{k,h}_\mathrm{loc}(\Omega)$ consists of $u\in C^k(\Omega)$ such that $u\in C^{k,h}(U)$ for every open $U\subset \subset \Omega$. 
We see with the aid of the partition of unity that a function $u:\Omega\to \C$ belongs to $W^{k,p}_\mathrm{loc}(\Omega)$ if and only if $u\in W^{k,p}(B)$ for every ball $B\subset\subset \Omega$.

We use $\lesssim$ to denote the inequality up to a constant whenever it is convenient.
Throughout what follows, $g:\Omega\times \C\to \C$ stands for a Carath\'eodory function, i.e., measurable with respect to the first variable and continuous with respect to the second one.

The Brezis--Kato theorem --- Thm 2.3 in \cite{Brezis_Kato} --- asserts that if $u\in W^{1,2}_\mathrm{loc}(\Omega)$ is a weak solution of 
\[
-\Delta u  = g(x,u)\ \text{in }\Omega,
\]
that is for every $\f\in C^\infty_0(\Omega)$ 
\[
\int_\Omega \nabla u \nabla \f \, dx = \int_\Omega g(x,u)\f\, dx
\]
and there exists a non-negative $a\in L^\frac{N}{2}_\mathrm{loc}(\Omega)$ such that
\[
|g(x,u(x))|\leq a(x)(1+|u(x)|)\ \text{a.e. in }\Omega,
\]
then $u\in \bigcap_{1\leq q<\infty} L^q_\mathrm{loc}(\Omega)$.
See also Lemma B.3 in Struwe's book \cite{Struwe} for a direct exposition of this result.

Notice that the regularity of solutions which the Brezis--Kato theorem asserts is higher than the one obtained through the elliptic regularity theory.
This has numerous important consequences, let us name just two.
The Brezis--Kato theorem has proved to be useful in the study of elliptic equations through the variational approach.
It enables one to show that critical points of certain functionals are classical $C^2$-solutions, see \cite{Struwe} and \cite{Willem}.
For another example, the Brezis--Kato theorem turns out to be crucial in proving the so-called Pohozaev identity, cf. \cite{Berestycki_Lions}.

The Brezis--Kato theorem has been extended to the case of the nonlocal right hand side in \cite{Moroz} and generalized to the case of the fractional Laplacian in \cite{Leite}.
The generalization to the bilaplacian operator has recently been proved in \cite{MS}.
To the best of the author’s knowledge these are the only results present in the literature.

In this paper we generalize the Brezis--Kato theorem in two directions.
Firstly, we consider the general elliptic differential operator with variable coefficients.
Secondly, we allow the elliptic operator to be of any even order.
Results presented here seems to be new even for the second order differential operator.

The proof of the main result Theorem \ref{thm} is inspired by \cite{Van_der_Vorst}.
It was developed and adjusted to the bilaplacian case in \cite{MS} and is presented here in a full generality.

The organization of the paper is the following.
In Section \ref{section:2} we provide the proof of the main result --- Theorem \ref{thm} --- divided into some number of lemmata.
Section \ref{section:3} contains Theorem \ref{thm:2} and Theorem \ref{thm:3} which are just versions of Theorem \ref{thm} adjusted to the weaker notions of solutions occurring in the literature.
\section{The Main Result}\label{section:2}

\begin{theorem}\label{thm}
Let $m \geq 1 $, $N > 2m $ be positive integers, $p>1$ and $\Omega\subset \R^N$ be open.
Let $ u \in W^{2m, p}_\mathrm{loc}(\Omega)$ be a \emph{strong solution} of $Lu = g(x,u)$, i.e.,
\begin{equation}\label{eq:13}
L u = g(x,u(x))\quad \text{a.e. in }\Omega,
\end{equation}
If there exists a non-negative $a\in L^\frac{N}{2m}_\mathrm{loc}(\Omega)$ such that
\begin{equation}\label{eq:1}
\begin{aligned}
 |g(x,u(x))| \leq a(x)(1 + |u(x)|), \quad\text{a.e. in } \Omega,
\end{aligned}
\end{equation}
then $u\in \bigcap_{1\leq q < \infty}L^q_\mathrm{loc}(\Omega)$.
If $g$ has a polynomial rate of growth, namely, there exists  $r\geq 1$ such that
\begin{equation}\label{poplynomial:1}
\begin{aligned}
|g(x,s)|\lesssim  1 + |s|^r\quad\text{for all }x\in \Omega\text{ and } s\in \C, \;\text{where }
r  =
\begin{cases}
\frac{N}{N-2mp} &\text{if } p<\frac{N}{2m},\\
\text{ is arbitary} &\text{if } p \geq \frac{N}{2m},
\end{cases} 
\end{aligned}
\end{equation}
then $u\in W^{2m,q}_\mathrm{loc}\cap C^{2m-1,h}_\mathrm{loc}(\Omega)$ for all $1\leq q < \infty$ and $0< h< 1$.
\end{theorem}

\begin{remark}
Observe that the hypothesis $N > 2m$ in Theorem \ref{thm} is the only interesting case, for otherwise the Sobolev embedding $W^{2m,1}_\mathrm{loc}(\Omega)\subset \bigcap_{1\leq q<\infty} L^q_\mathrm{loc}(\Omega)$ immediately implies the assertion for any strong solution.
Form now on we confine our attention to $N>2m$ without further comment.

Let us also emphasize that without stronger assumptions Theorem \ref{thm} is false in general for $p=1$ , as there is a counterexample, see Theorem 3 in \cite{escauriaza}.
\end{remark}

Before proving Theorem \ref{thm}, we need some preparations first.

\begin{proposition}\label{prop2}
Let $s \geq \frac{N}{N-2m}$.
If for some $u\in L^s_\mathrm{loc}(\Omega)$ there is a non-negative $a\in L^\frac{N}{2m}_\mathrm{loc}(\Omega)$ such that
\[
|g(x,u(x))|\leq a(x) (1 +  |u(x)|), \quad \text{a.e. in }\Omega,
\]
then $g(\cdot,u) \in L^\frac{Ns}{N+2ms}_\mathrm{loc}(\Omega)$.
\end{proposition}
\begin{proof}
For any ball $B\subset\subset \Omega$, we have
\[
\int_B |g(x,u)|^\frac{Ns}{N+2ms}\, dx \lesssim \int_B a^\frac{Ns}{N+2ms}\, dx + \int_B \le a |u|\pr ^\frac{Ns}{N+2ms}\, dx.
\]
Since $\frac{Ns}{N+2ms}<\frac{N}{2m}$, the first integral $\int_B a ^\frac{Ns}{N+2ms}\, dx$ is finite.
We use the H\"older inequality with the exponents
\[
\frac{1}{\frac{N+2ms}{2ms}} + \frac{1}{\frac{N+2ms}{N}}=1
\]
to get 
\[
\int_B \le a |u|\pr ^\frac{Ns}{N+2ms}\, dx \leq \le \int _B a^\frac{N}{2m}\, dx \pr ^\frac{2ms}{N+2ms} \le \int_B |u|^s \, dx\pr^\frac{N}{N+2ms} < \infty.
\]
Therefore $\int_B |g(x,u)|^\frac{Ns}{N+2ms}\, dx$ is finite what completes the proof.
\end{proof}

The following lemma is based on Lemma B.2 from \cite{Van_der_Vorst}.
\begin{lemma}\label{lem:1}
Let $1\leq p< \infty$ and $U\subset \R^N$ be a measurable set with $0<|U|<\infty$, where $|U|$ denotes the measure of $U$.
We assume that $a\in L^p(U)$ and $v\in L^1(U)$.
For every $\e>0$ there exists $q_\e \in L^p(U)$ and $f_\e\in L^\infty(U)$ such that $\|q_\e\|_{L^p(U)} < \e$ and
\[
 a (x) v(x) = q_\e(x) v(x) +  f _\e(x) \quad\text{for a.e. }x\in U.
\] 
\end{lemma}
\begin{proof}
We define the sets
\[
A_k = \left\{ x\in U\mid |a(x)|<k\right\}\quad\text{and}\quad B_k=\left\{x\in U\mid|v(x)|<k\right\}.
\]
Since $a$, $v\in L^1(U)$, we have $|A_k^c|$, $|B_k^c|\to 0$. 
Observe that
\[
\begin{aligned}
|A_k^c \cup B_k^c| = |A_k^c| + |B_k^c| - |A_k^c\cap B_k^c| \leq |A_k^c| + |B_k^c|\to 0,
\end{aligned}
\]
so there is $k_1$ such that for $k\geq k_1$
\[
|A_k\cap B_k| = |U|- |A_k^c\cup B_k^c| > 0.
\]
Fix $\e>0$.
By the absolute continuity of the measure $\mu(C) := \int_{C}|a(x)|^p\d x$, we find $k_2\geq k_1$ such that 
\begin{equation}\label{eq:3}
\int_{A_k^c \cup B_k^c}|a(x)|^p\d x < \le \frac{1}{2}\e\pr ^{p}, \quad \text{for }k\geq k_2.
\end{equation}
We are in a position to define 
\[
\begin{aligned}
q_\e(x) : &=
\begin{cases}
\frac{1}{k}a(x), &\text{for }x\in A_k\cap B_k,\\
a(x), &\text{for }x\in (A_k\cap B_k)^c,
\end{cases}
\quad\text{and}\quad  f_\e &:= (a(x)-q_\e(x))v(x).
\end{aligned}
\]
We estimate using the definition of $q_\e$ and \eqref{eq:3}, for $k\geq k_2$,
\[
\begin{aligned}
\|q_\e\|_{L^p(U)} &\leq  \frac{1}{k}\le \int_{A_k\cap B_k} |a(x)|^p \d x\pr ^{1/p} + \le \int_{(A_k\cap B_k)^c} |a(x)|^p\pr^{1/p} \d x < \frac{1}{k}\|a\|_{L^p(U)}+ \frac{1}{2}\e.
\end{aligned}
\]
If $k\geq \max\left\{\frac{2\|a\|_{L^p(U)}}{\e}, k_2\right\}$, then
\[
\|q_\e\|_{L^p(U)} < \e.
\]
Observe that $f_\e = 0$ on $(A_k\cap B_k)^c$, hence we get
\[
|f_\e(x)|= \left |a(x) - \frac{1}{k}a(x)\right ||v(x)|\1_{A_k\cap B_k}(x)= \frac{k-1}{k}|a(x)||v(x)|\1_{A_k\cap B_k}(x)\leq (k-1)k,
\]
in view of the definitions of $A_k$ and $B_k$.
\end{proof}

The next lemma is crucial. It is a modification of Lemma 2.2 in \cite{MS}.

\begin{lemma}\label{lem:2}
If $u\in W^{2m,p}_\mathrm{loc}(\Omega)$, with $p > 1$, is a strong solution of 
\[
Lu = g(x,u)\quad \text{in }\Omega,
\]
and \eqref{eq:1} holds, then 
\[
u\in 
\begin{cases}
L^\frac{Np}{N-(2m+1)p}_\mathrm{loc}(\Omega), &\text{if }p < \frac{N}{2m+1},\\
\bigcap_{1\leq q<\infty}L^q_\mathrm{loc}(\Omega) , &\text{if }p\geq \frac{N}{2m+1}.
\end{cases}
\]
\end{lemma}
\begin{proof}
Observe that if $2mp\geq N$, then we are done due to the Sobolev embedding $W^{2m,p}_\text{loc}(\Omega)\subset  L^q_\text{loc}(\Omega)$ for  $1\leq q<\infty$.
We therefore assume that $2mp < N$.

Take some ball $B_0\subset \subset \Omega$ and choose another ball $B$ such that $ B_0\subset\subset  B \subset\subset \Omega$.
Take a smooth cut-off function $\eta \in C^\infty_0(B)$ satisfying $\eta \equiv 1$ on $ B_0$ and $0\leq \eta \leq 1$.
Let us define 
\[
\begin{aligned}
\tilde{a}(x) &:= \frac{g(x,u(x))}{u(x)}\1 _{\left\{ x\mid |u(x)|> 1\right\}}(x),\\
b(x) &:=  g(x, u(x))\1_{\left\{ x\mid |u(x)|\leq 1\right\}}(x),
\end{aligned}
\]
for $x\in B$, where it is understood that $\tilde a (x)=0$ whenever $u(x) = 0 $.
We then have
\[
g(x,u(x)) = \tilde{a}(x)u(x) + b(x), \quad x\in B,
\]
and by \eqref{eq:1}
\[
\begin{aligned}
|\tilde{a}(x)| &\leq \frac{a(x)(1+|u(x)|)}{|u(x)|} \1 _{\left\{ x\mid |u(x)|> 1\right\}}(x)\leq \frac{a(x)(2|u(x)|)}{|u(x)|} \1 _{\left\{ x\mid |u(x)|> 1\right\}}(x) \leq 2a(x),\\
|b(x)| &\leq a(x)(1+|u(x)|)\1_{\left\{ x\mid |u(x)|\leq 1\right\}}(x)\leq 2 a(x),
\end{aligned}
\]
so $\tilde a$, $b \in L^\frac{N}{2m}(B)$.
Moreover, $u$ is a strong solution of
\begin{equation}\label{eq:6}
L u = \tilde{a}(x)u + b(x)\quad\text{in }B.
\end{equation}

Let us consider 
\[
v := u\eta \in W^{2m,p}(B)\cap W^{m,p}_0(B).
\]
By the Leibniz formula for any multi-index $|\alpha|\leq 2m$ we get
\[
D^\alpha v= \le D^\alpha u \pr \eta +
\sum_{\beta < \alpha} \binom{\alpha}{\beta}D^\beta u D^{\alpha-\beta} \eta,
\]
and so 
\begin{equation}\label{eq:5}
\begin{aligned}
L v &=(-1)^m\sum_{|\alpha|\leq 2m}a_\alpha(x) D^\alpha v = \sum_{|\alpha|\leq 2m}a_\alpha(x) \le \le D^\alpha u \pr \eta +
\sum_{\beta < \alpha} \binom{\alpha}{\beta}D^\beta u D^{\alpha-\beta} \eta\pr \\
&= \le L u\pr \eta + F(u, \eta),
\end{aligned}
\end{equation}
where 
\[
F(u,\eta)=\sum_{\beta < \alpha,\, |\alpha| = 2m} c_{\alpha,\beta}(x)D^\beta u D^{\alpha-\beta} \eta,
\]
for some coefficients $c_{\alpha,\beta}\in L^\infty(B)$.
Since $u\in W^{2m,p}(B)\subset W^{2m-1,\frac{Np}{N-p}}(B)$, we estimate
\begin{equation}\label{eq:2}
\left\| F(u,\eta)\right \|_{L^\frac{Np}{N-p}(B)} \lesssim  \|u\|_{W^{2m-1,\frac{pN}{N-p}}(B)}\|\eta\|_{W^{2m,\infty}(B)} \lesssim \|u\|_{W^{2m,p}(B)}.
\end{equation}
By Lemma \ref{lem:1}, for every $\e>0$ there is $q_\e\in L^\frac{N}{2m}(B)$ and $\wh f_\e\in L^\infty(B)$ such that
\begin{equation}\label{eq:4}
\tilde a(x) v(x) = q_\e(x) v(x) + \wh f_\e(x)\quad\text{and}\quad \|q_\e\|_{L^\frac{N}{2m}(B)}< \e.
\end{equation}
We use \eqref{eq:5}, \eqref{eq:6} and \eqref{eq:4}
\[
\begin{aligned}
L v &=  \le \tilde{a}(x)u + b(x)\pr \eta +F(u,\eta)= \tilde{a}(x)v + b(x)\eta + F(u, \eta)\\
&= q_\e(x)v  + \underbrace{\wh f_\e(x) + b(x)\eta}_{f_\e} + F(u,\eta) = q_\e(x)v  + f_\e(x) +F(u,\eta),
\end{aligned}
\]
with
\begin{equation}\label{eq:7}
f_\e \in L^\frac{N}{2m}(B).
\end{equation}
We therefore arrive at 
\begin{equation}\label{eq:23}
Lv - q_\e v = f_\e + F(u,\eta).
\end{equation}

\begin{proposition}\label{observation}
Let $1< p < q <\frac{N}{2m}$ and $v\in W^{2m,p}(B)\cap W^{m,p}_0(B)$.
If for every $\e>0$ we have $Lv - q_\e v \in L^q(B)$ whith $\|q_\e\|_{L^\frac{N}{2m}(B)}<\e$, then $v\in W^{2m,q}(B)$.
\end{proposition}
\begin{proof}

We recall some theory about differential and unbounded operators, see \cite{Browder} and \cite{Kato} for details and proofs.
Fix $1<s<\infty$.
Let $A_s$ be the $L^s(B)$--realization of the elliptic differential operator $L = (-1)^m\sum_{|\alpha|\leq 2m}a_\alpha(x)D^\alpha$ satisfying (i) and (ii) under null Dirichlet boundary conditions.
The domain $D_s$ of $A_s$ equals to $W^{2m,s}(B)\cap W^{m,s}_0(B)$.
For $u\in D_s$ we set $A_s u = Lu = (-1)^m\sum_{|\alpha|\leq 2m}a_\alpha(x)D^\alpha u$ and the latter is a well-defined element of $L^s(B)$.
There exists a constant $\tilde{k}_s>0$ such that for every $u\in D_s$
\begin{equation}\label{eq:12}
\|u\|_{W^{2m,s}(B)}\leq \tilde{k}_s \le \|Lu\|_{L^s(B)} +  \|u\|_{L^s(B)}\pr,
\end{equation}
see formula (1) in \cite{Browder}.

By Thm 2 of \cite{Browder}, the inverse operator $(A_s+ \zeta I)^{-1}:L^s(B)\to D_s\subset L^s(B)$ exists (notice that $ D(A_s + \zeta I) =D_s$)  for large positive $\zeta$.
We fix such a large $\zeta>0$ and consider the invertible operator $\tilde{A}_s := A_s + \zeta I$.
For $f\in L^s(B)$ we set $u =  \tilde{A}_s^{-1}f \in D_s$.
Then 
\[
\|u\|_{L^s(B)} \leq \|\tilde{A}_s^{-1}\|\|f\|_{L^s(B)} = c_s\|\tilde{A}_s u\|_{L^s(B)}
\]
The above and the inequality \eqref{eq:12} applied to the operator $L+\zeta I$ yields
\begin{equation}\label{eq:21}
\|u\|_{W^{2m,s}(B)}\leq  k_s\|\tilde A _s u\|_{L^s(B)},\quad u \in D_s.
\end{equation}

Let us also consider for $1<s<\frac{N}{2m}$ the multiplication operator $Q_{\e,s}$ considered as an unbounded operator from $L^s(B)$ into $L^s(B)$ defined by 
\[
Q_{\e,s} u = q_\e u
\] 
for $u$ in the domain $D(Q_{\e,s})$.
Observe that for any $u\in L^\frac{Ns}{N-2ms}(B)$ we have
\begin{equation}\label{eq:22}
\|Q_{\e,s}\|_{L^s(B)} \leq \|q_\e\|_{L^\frac{N}{2m}(B)}\|u\|_{L^\frac{Ns}{N-2ms}(B)} ,
\end{equation}
where we used the H\"older inequality with the exponents
\[
\frac{1}{s} = \frac{1}{\frac{N}{2m}} + \frac{1}{\frac{Ns}{N-2ms}}.
\]
For $u\in D_s$  we utilize \eqref{eq:22}, the Sobolev embedding $W^{2m,s}(B)\subset L^\frac{Ns}{N-2ms}(B)$ and \eqref{eq:21} to deduce
\[
\|Q_{\e,q}u\|_{L^q(B)} \leq c_\mathrm{Sobolev}\|q_\e\|_{L^\frac{N}{2m}(B)} \|u\|_{W^{2m,q}(B)} \leq c_\mathrm{Sobolev} k_ s \|q_\e\|_{L^\frac{N}{2m}(B)} \|\tilde A_s u\|_{L^q(B)}.
\]
Therefore $Q_{\e,s}$ is relatively bounded with respect to $\tilde{A}_s$, see \S 4 of \cite{Kato} for the definitions and details.
If $\|q_\e\|_{L^\frac{N}{2m}(B)}$ is small enough so that $\|q_\e\|_{L^\frac{N}{2m}(B)} c_\mathrm{Sobolev}k_s<1$, then $(\tilde{A}_s - Q_{\e,s})^{-1}:L^s(B) \to D_s \subset L^s(B)$ exists in view of Thm 1.16 in \S 4.4 of \cite{Kato}.
By \eqref{eq:4} we can always assume this is the case.

\textit{Case I:} $q \leq \frac{Np}{N-2mp}$

In this case $v \in W^{2m,p}(B) \subset L^\frac{Np}{N-2mp}(B)\subset L^q(B)$ and so
\[
\underbrace{Lv - q_\e v }_{\in  L^q(B)} +\underbrace{\zeta v}_{\in L^q(B)} =:f \in L^q(B),
\]
where $\zeta $ is large enough so that  both $\tilde A_q^{-1}$ and $\tilde A_p ^{-1}$ exist.
We define 
\[
y:= (\tilde{A}_q - Q_{\e,q})^{-1}f\in D_q,
\]
then
\[
(\tilde{A}_q  - Q_{\e,q})y = f.
\]
On the other hand, we notice that $v\in D_p$ so $(\tilde{A}_p  - Q_{\e,p})v$ is well-defined and equals to $f$.
Since $q>p$ and $D_q\subset D_p$, we may consider $\tilde{A}_q$ as a restriction of $\tilde{A}_p$ and thus
\[
(\tilde{A}_q  - Q_{\e,q})y = (\tilde{A}_p  - Q_{\e,p})y.
\]
This yields
\[
(\tilde{A}_p  - Q_{\e,p})(y-v) =  0.
\]
Since $\tilde{A}_p  - Q_{\e,p}$ is invertible, we get 
\[
v = y \in D^q\subset W^{2m,q}(B)
\]
as claimed.

\textit{Case II:} $q > \frac{Np}{N-2mp}$

We proceed similarly as in the previous case, but this time
\[
Lv  - q_\e v + \zeta v \in L^\frac{Np}{N-2mp}(B), \quad \zeta >0,
\]
and, reasoning as above, we can conclude that $v\in D_\frac{Np}{N-2mp}$.
We define $p_i := \frac{Np}{N-2mip}$ for $i = 0,\ldots,k+1$, where $k\geq 1$ is the integer satisfying
\[
k < \frac{N(q-p)}{2mpq}\leq k+1.
\]
Repeating this procedure $k$ times we see that $v\in D_{p_k}$ where $p_k$  satisfies  $p_k < q \leq p_{k+1}$, and so $v\in W^{2m,p_k}(B)\subset L^{p_{k+1}}(B) \subset L^q(B)$.
We then have 
\[
Lv - q_\e v + \zeta v \in L^q(B), \quad \zeta > 0,
\]
and proceeding as in \textit{Case I} we show that $u\in W^{2m,q}(B)$.
\end{proof}

We continue the proof of Lemma \ref{lem:2} by considering two cases.
If $p<\frac{N}{2m+1}$, then for $q =  \frac{Np}{N-p}$ we estimate using \eqref{eq:2}
\[
\|f_\e + F(u,\eta)\|_{L^q(B)} \lesssim \|f_\e\|_{L^\frac{N}{2m}(B)}  + \|u\|_{W^{2m,p}(B)},
\]
By \eqref{eq:7} and the hypothesis the $L^q(B)$ norm of the right hand side of \eqref{eq:23} is finite.
By Proposition \ref{observation} $v\in W^{2m,q}(B) \subset L^\frac{Np}{N-(2m+1)p}(B)$.
Since $v = u$ on $B_0$ and $B_0$ was arbitrary, we infer that $u\in L^\frac{Np}{N-(2m+1)p}_\mathrm{loc}(\Omega)$.

If $p\geq \frac{N}{2m+1}$, we proceed in a similar fashion as in Case I.  
Fix any $q \geq \frac{N}{N- 2mp}$ and define $r:= \frac{Nq}{N+2mq}$.
Then we get the following relations 
\[
1 < r <\frac{N}{2m}\leq \frac{Np}{N-p}.
\] 
We estimate as above using \eqref{eq:2}
\[
\|f_\e + F(u,\eta)\|_{L^r(B)} \lesssim \|f_\e\|_{L^\frac{N}{2m}(B)}  + \|u\|_{W^{2m,p}(B)}.
\]
We observe that the right hand side of \eqref{eq:23} has finite $L^r(B)$ norm by \eqref{eq:7} and the hypothesis.
Proposition \ref{observation} yields $v\in W^{2m,r}(B) \subset L^q(B)$.
Once again it implies that $u\in L^q(B_0)$.
Since  $B_0\subset \Omega$ and $q\geq \frac{N}{N-2mp}$ were arbitrary $u\in \bigcap_{1\leq q < \infty} L^q_\mathrm{loc}(\Omega)$.
\end{proof}

We need the following observation which bears great resemblance to Proposition \ref{observation}.
We present the proof for the sake of the clarity.
\begin{proposition}\label{prop}
Let $u\in W^{2m,p}_\mathrm{loc}(\Omega)$ satisfy $Lu \in L^q_\mathrm{loc}(\Omega)$ for some $1< p<q<\infty$.
Then $u\in W^{2m,q}_\mathrm{loc}(\Omega)$.
\end{proposition}
\begin{proof}
The proof is similar in spirit to that of Proposition \ref{observation}.
Fix $B_0\subset\subset \Omega$ and choose another ball $B$ such that $B_0\subset\subset B\subset\subset \Omega$.
Let $\eta\in C^\infty_0(B)$ be smooth cut--off function such that $0\leq \eta\leq 1$ and $\eta=1$ on $B_0$.
We consider the function 
\[
v := \eta u \in W^{2m,p}(B)\cap W^{m,p}_0(B)
\]
which satisfies (as in \eqref{eq:5})
\[
Lv = (Lu)\eta + \sum_{\beta < \alpha,\, |\alpha| = 2m} c_{\alpha,\beta}(x)D^\beta u D^{\alpha-\beta} \eta = (L u)\eta + F(u,\eta).
\]
We denote by $\tilde A_s$ the $L^s(B)$--realization of $L +\zeta I$.
Observe we only need to show that $v\in W^{2m,q}(B)$.

We consider two cases.
If $W^{2m,p}(B)\subset W^{2m-1,q}(B)$ (that is either $p>N$ or $q\leq \frac{Np}{N-p}$), then
\[
\|F(u,\eta)\|_{L^q(B)}\lesssim \|u\|_{W^{2m-1,q}(B)}\|\eta\|_{W^{2m,\infty}(B)}\lesssim\|u\|_{W^{2m,p}(B)},
\]
what yields
\[
Lv\in L^q(B).
\]
We choose $\zeta >0$ large enough so that both $\tilde{A}_p^{-1}$ and $\tilde A_q^{-1}$ exist and set $f:= Lv +\zeta v \in L^q(B)$.
We consider
\[
y = A_q^{-1}f \in D_q = W^{2m,q}\cap W^{m,q}_0(B)\subset D_p.
\]
Then we have $y - v \in D_p$ and 
\[
\tilde A_p(y - v) =(L + \zeta )(y - v) = 0
\]
Since $\tilde A_p:D_p \to L^p(B)$ is a bijection, we obtain $y = v$ and so $v \in W^{2m,q}(B)$.

If $W^{2m,p}(B) \not\subset W^{2m-1,q}(B)$ (that is $p<N$ and $q > \frac{Np}{N-p}$), then as in \eqref{eq:5}
\[
\|F(u,\eta)\|_{L^\frac{Np}{N-p}(B)}\lesssim\|u\|_{W^{2m,p}(B)}.
\]
Hence
\[
Lv\in L^\frac{Np}{N-p}(B)
\]
and, similarly as above, this implies that $v\in D_\frac{Np}{N-p}$.
After finite number of such steps we find $s < N$ such that $v\in D_s\subset W^{2m,s}(B) \subset L^q(B)$.
Then the same argument as in the previous case applies to show that $v\in W^{2m,q}(B)$.

Summarizing, in both cases $v\in W^{2m,q}(B)$ and the proof is completed.
\end{proof}

\begin{proof}[Proof of Theorem \ref{thm}]
Let $p>1$ and $u\in W^{2m,p}_\mathrm{loc}(\Omega)$.
If $p\geq \frac{N}{2m+1}$, then $u\in \bigcap_{1\leq q<\infty }L^q_\mathrm{loc}(\Omega)$ due to Lemma \ref{lem:2}.
We thus assume that $p < \frac{N}{2m+1}$.

First, we show that whenever $u\in W^{2m,s}_\text{loc}(\Omega)$ with some $1<s<\frac{N}{2m+1}$ is the strong solution of 
\[
Lu = g(x,u) \quad\text{in }\Omega,
\]
then  $u\in W^{2m,\frac{Ns}{N-s}}_\text{loc}(\Omega)$.
Indeed, Lemma \ref{lem:2} implies that $u\in L^\frac{Ns}{N-(2m+1)s}_\mathrm{loc}(\Omega)$ and $g(\cdot,u)\in L^\frac{Ns}{N-s}_\mathrm{loc}(\Omega)$ in view of Proposition \ref{prop2}.
Since $Lu = g(x,u)\in  L^\frac{Ns}{N-s}_\mathrm{loc}(\Omega)$, Proposition \ref{prop} yields $u\in W^{2m,\frac{Ns}{N-s}}_\mathrm{loc}(\Omega)$ as claimed.

Let $k\geq 1$ be the integer satisfying
\[
k -1 < \frac{N-(2m+1)p}{p} \leq k
\]
and define
\[
p_i := \frac{Np}{N-ip}, \quad\text{for } i=0,\ldots, k.
\]
Notice that $k$ is chosen in such a way that $p_{k-1} < \frac{N}{2m+1}$ and $p_k\geq \frac{N}{2m+1}$.
The first part of the proof shows that if $u\in W^{2m,p_i}_\mathrm{loc}(\Omega)$, then $u\in W^{2m,p_{i+1}}_\mathrm{loc}(\Omega)$.
After $k$ such steps we obtain $u\in W^{2m,p_{k}}_\mathrm{loc}(\Omega)$ and Lemma \ref{lem:2} yields the assertion.

We now proceed with the proof of the second assertion. Let $u\in W^{2m,p}_\mathrm{loc}(\Omega)$ with $p>1$ be the strong solution of $Lu=g(x,u)$ in $\Omega$ with $g$ of polynomial growth \eqref{poplynomial:1}.
If $p\geq \frac{N}{2m}$, then the Sobolev embedding implies that $u\in \bigcap_{1\leq 1< \infty}L^q_\mathrm{loc}(\Omega)$.
Hence  the superposition $g(\cdot,u) \in L^q_\mathrm{loc}(\Omega)$ for any $q\geq 1$ and Proposition \ref{prop} yields $u\in W^{2m,q}_\mathrm{loc}(\Omega)$.
Then $u\in C^{2m-1,h}_\mathrm{loc}(\Omega)$ due to the Sobolev embeddings for any $0<h<1$.

We now consider the case $p<\frac{N}{2m}$.
We claim that there exists $a\in L^\frac{N}{2m}_\mathrm{loc}(\Omega)$ such that 
\[
|g(x,u(x))|\leq a(x)(1+|u(x)|),\quad \text{a.e. in }\Omega.
\]
Indeed, defining
\begin{equation}\label{eq:20}
a(x) = \frac{|g(x,u(x))|}{1+|u(x)|}
\end{equation}
and fixing some $B\subset\subset \Omega$ we have by \eqref{poplynomial:1}
\[
\begin{aligned}
\int_B a^\frac{N}{2m}\d x  &  \lesssim \int_B  \le \frac{1 + |u|^\frac{N}{N-2mp}}{1 + |u|}\pr ^\frac{N}{2m} \d x  \\
&\leq  \int_{B \cap \left\{x\mid |u(x)|\leq 1 \right\} }2^\frac{N}{2m}\d x + \int_{B \cap\left\{ x\mid |u(x)| >1 \right\}} \le \frac{2|u|^\frac{N}{N-2mp}}{|u|}\pr ^\frac{N}{2m}\d x \\
&\lesssim  1 + \int_B{|u|^{\frac{2mp}{N-2mp}\frac{N}{2m}}}\d x < \infty,
\end{aligned}
\]
since $u\in W^{2m,p}(B)\subset L^\frac{Np}{N-2mp}(B)$.
We apply the already proved first part of Theorem \ref{thm} to infer that $u\in \bigcap_{1\leq q < \infty} L^q_\mathrm{loc}(\Omega)$.
This and the polynomial rate of growth \eqref{poplynomial:1} enable us to conclude that $L u  = g(\cdot,u)\in L^q_\mathrm{loc}(\Omega)$ for any $q\geq 1$.
Once again Proposition \ref{prop} and the Sobolev embeddings finish the proof.
\end{proof}

\section{(Very) Weak Solutions}\label{section:3}
This section is devoted to the Brezis--Kato theorem for certain types of weak solutions.
We make stronger assumptions on the coefficients of the differential operator $L$ to deal with less regular solutions.
We consider the differential operator in the divergence form as well. 

Let us recall that $L$ is a strongly elliptic linear differential operator satisfying the hypothesis stated in Introduction.
We additionally assume here that the coefficients $a_\alpha$ of $L$ are more regular, namely,
\begin{equation}\label{coeff_reg}
a_\alpha \in C^{|\alpha|}(\Omega),\quad\text{for}\; 0<|\alpha|\leq 2m.
\end{equation}
Let $f\in L^1_\text{loc}(\Omega)$.
We say that $u\in L^1_\text{loc}(\Omega)$ is a \emph{very weak solution} of 
\[
Lu  = f,
\]
if for every $\f\in C^\infty_0(\Omega)$
\[
\int_\Omega  u(x)\sum_{|\alpha|\leq 2m} (-1)^{m+|\alpha|}  D^\alpha\le a_\alpha(x)\f(x)\pr \, d x = \int_\Omega f(x)\f(x)\, dx.
\]
The above formula arises in a usual manner from formally multiplying the differential equation $Lu=f$ by a test function $\f$, integrating over $\Omega$ and using integration by parts.

\begin{lemma}\label{lem_weak}
Let $1<p,\,q<\infty$.
If $u\in L^q_\mathrm{loc}(\Omega)$ is a very weak solution of $Lu = f\in L^p_\mathrm{loc}(\Omega)$, then $u\in W^{2m,p}_\mathrm{loc}(\Omega)$.
\end{lemma}
\begin{proof}
Take any ball $B_0\subset\subset \Omega$ and choose another ball $B$ such that $B_0\subset\subset B\subset \subset \Omega$.
We introduce the adjoint $A$ of $L$ in $B$ by the formula
\[
A u :=\sum_{|\alpha|\leq 2m} (-1)^{m+ |\alpha|}D^\alpha(\overline{a_\alpha(x)}u).
\]
Due to the Leibniz formula and \eqref{coeff_reg} the operator $A$ is itself an elliptic differential operator, i.e., $A u = \sum_{|\alpha| \leq 2m } b_\alpha(x)D^\alpha u$ for some coefficients $b_\alpha$.
Note that $A$ and $u$ fulfill all the hypotheses of Thm 7.1 in~\cite{Agmon} and therefore $u\in W^{2m,p}_\mathrm{loc}(B)\subset W^{2m,p}(B_0)$.
This finishes the proof.
\end{proof}

\begin{theorem}\label{thm:2}
Let  $L$ by strongly elliptic operator with the coefficients satisfying \eqref{coeff_reg}.
Let $u\in L^p_\mathrm{loc}(\Omega)$ for some $p > \frac{N}{N-2m}$ be a very weak solution of
\begin{equation}\label{eq:24}
Lu = g(x,u)\quad \text{in }\Omega.
\end{equation}
If $g$ satisfies \eqref{eq:1}, then $u \in \bigcap_{1\leq q<\infty} L^q_\mathrm{loc}(\Omega)$.
If $g$ has a polynomial rate of growth, namely, 
\begin{equation}\label{polynomial:2}
|g(x,s)|\lesssim 1+|s|^\frac{N+2mp}{N},\quad \text{for all }x\in \Omega\text{ and } s\in  \C,
\end{equation}
then $u\in W^{2m,q}_\mathrm{loc}(\Omega)\cap C^{2m-1,h}_\mathrm{loc}(\Omega)$, for all $1\leq q< \infty$ and $0< h < 1$.
\end{theorem}
\begin{proof}
Let $u\in L^p_\mathrm{loc}(\Omega)$ be a very weak solution of $Lu = g(x,u)$.
By Proposition \ref{prop2}, $Lu = g(\cdot,u)\in L^\frac{Np}{N+2mp}_\mathrm{loc}(\Omega)$ and $u\in W^{2m,\frac{Np}{N+2mp}}_\mathrm{loc}(\Omega)$ in view of Lemma \ref{lem_weak}.
Such a regularity implies that $u$ is a strong solution of \eqref{eq:24} and Theorem \ref{thm} yields $u\in \bigcap_{1\leq q < \infty} L^q_\mathrm{loc}(\Omega)$.

If $g$ satisfies \eqref{polynomial:2}, then we define 
\[
a(x)= \frac{|g(x,u(x))|}{1 + |u(x)|}
\]
and proceed similarly as in \eqref{eq:20} and below to show that $a\in L^\frac{N}{2m}_\mathrm{loc}(\Omega)$ and  $g$ satisfies \eqref{eq:1}. 
By the first part of the proof $u\in \bigcap_{1\leq q< \infty}L^q_\mathrm{loc}(\Omega)$ and the polynomial rate of growth \eqref{polynomial:2} then implies that $L u = g(\cdot,u) \in \bigcap_{1\leq q< \infty}L^q_\mathrm{loc}(\Omega)$.
We deduce utilizing Lemma \ref{lem_weak} that $u\in W^{2m,q}_\mathrm{loc}(\Omega)$ for every $q\geq 1$.
The last assertion follows by the Sobolev embeddings $W^{2m,q}_\mathrm{loc}(\Omega) \subset C^{2m-1,h}_\mathrm{loc}(\Omega)$.
\end{proof}

We now consider a differential operator in a \emph{divergence form}
\begin{equation}\label{divergence}
Lu = (-1)^m\sum_{|\alpha|,\, |\beta| \leq m}D^\beta \le a_{\alpha\beta}(x)D^\alpha u\pr,
\end{equation}
where $a_{\alpha\beta}\in C^{|\alpha| + |\beta|}(\Omega)$.
We say that $L$ is elliptic provided that
\[
\mathfrak{Re}\le\sum_{|\alpha| = |\beta | = 2m} a_{\alpha\beta}(x)\xi^{\alpha+\beta}\pr  \geq \lambda |\xi|^{2m}, \quad \text{for }\xi \in \R^N.
\]
Let $f\in L^1_\mathrm{loc}(\Omega)$.
We say that $u\in W^{m,1}_\mathrm{loc}(\Omega)$ is a \emph{weak solution} of 
\[
Lu = f\quad \text{in }\Omega,
\]
if for every test function $\f\in C^\infty_0(\Omega)$
\[
\int_\Omega\sum_{|\alpha|,\, |\beta| \leq m} (-1)^{m+|\beta|} a_{\alpha\beta}(x)D^\alpha u(x) D^\beta \f(x)\, dx = \int_\Omega f(x)\f(x)\, dx.
\]

\begin{theorem}\label{thm:3}
Let $L$ be the elliptic differential operator in the divergence form \eqref{divergence} with the coefficients $a_{\alpha\beta}\in C^{|\alpha|+|\beta|}(\Omega)$.
Let $u\in W^{m,p}_\mathrm{loc}(\Omega)$ with $p> \frac{N}{N-m}$ be a weak solution of 
\[
Lu=g(x,u)\quad\text{in }\Omega.
\]
If $g$ satisfies \eqref{eq:1}, then $u\in \bigcap_{1\leq q<\infty} L^q_\mathrm{loc}(\Omega)$.
If $g$ has the polynomial rate of growth: there is $r > 0$ such that  
\begin{equation}\label{polynomial:3}
|g(x,s)|\lesssim  1+|s|^r,\quad \text{for all } x\in \Omega\text{ and } s\in \C, \text{ where }r =
\begin{cases}
\frac{N+mp}{N-mp},&\text{if }p < \frac{N}{m},\\
\text{is arbitrary},&\text{if }p \geq \frac{N}{m},
\end{cases}
\end{equation}
then $u\in W^{2m,q}_\mathrm{loc}(\Omega)\cap C^{2m-1,h}_\mathrm{loc}(\Omega)$ for all $1\leq q< \infty$ and $0< h < 1$.
\end{theorem}
\begin{proof}
If $p \geq \frac{N}{m}$, then $u\in \bigcap_{1\leq q<\infty} L^q_\mathrm{loc}(\Omega)$ due to the Sobolev embedding.
Hence we assume that $p<\frac{N}{m}$.
Fix $\f \in C^\infty_0(\Omega)$ in the definition of a weak solution.
Using integration by parts and ordering the terms one may show that $u$ is also a very weak solution to some differential operator in the non-divergence form $\tilde L u  = (-1)^m\sum_{|\alpha|\leq 2m}c_\alpha(x)D^\alpha u = g(x,u)$ in $\Omega$.
The smoothness assumption on the coefficients $a_{\alpha\beta}\in C^{|\alpha|+|\beta|}(\Omega)$ implies that $c_\alpha\in C^{|\alpha|}(\Omega)$.
Moreover, we have $\sum_{|\alpha|=2m}c_\alpha(x)\xi^\alpha = \sum_{|\alpha|=|\beta|=m}a_{\alpha\beta}(x)\xi^{\alpha+\beta}$, so $\tilde L$ is elliptic as well.
The Sobolev embedding $W^{m,p}_\mathrm{loc}(\Omega)\subset L^\frac{Np}{N-mp}_\mathrm{loc}(\Omega)$ where $\frac{Np}{N-mp} > \frac{N}{N-2m}$ and Theorem \ref{thm:2} proves the assertion.

Let us now assume that \eqref{polynomial:3} holds.
If $p\geq \frac{N}{m}$, then the Sobolev embedding yields $u\in \bigcap_{1\leq q < \infty}L^q_\mathrm{loc}(\Omega)$.
If $p < \frac{N}{m}$, then in view of \eqref{polynomial:3} $g$ satisfies \eqref{eq:1} with $a(x) := \frac{|g(x,u)|}{1 + |u|}$ what may be shown as above.
Therefore, in any case we obtain $u\in \bigcap_{1\leq q<\infty} L^q_\mathrm{loc}(\Omega)$.
The polynomial rate of growth implies that $g(\cdot,u)\in \bigcap_{1\leq q<\infty} L^q_\mathrm{loc}(\Omega)$, so $\tilde Lu \in \bigcap_{1\leq q<\infty} L^q_\mathrm{loc}(\Omega)$.
Lemma \ref{lem_weak} and the Sobolev embeddings finishes the proof.
\end{proof}

We present a particular form of Theorem \ref{thm:3} for polyharmonic operators in a (most common) $L^2$ setting for future references.
\begin{corollary}
Let $L$ be the polyharmonic operator $(-\Delta)^m$ defined on an open set $\Omega\subset \R^N$ with $N > 2m$.
Let $u\in W^{m,2}_\mathrm{loc}(\Omega)$ be a weak solution of
\[
(-\Delta)^mu = g(x,u) \quad\text{in }\Omega,
\]
in the sense that for every $\f\in C^\infty_0(\Omega)$
\[
\begin{aligned}
\int_\Omega \Delta^k u \Delta^k \f \, dx = \int_\Omega g(x,u) \f\, dx, \quad&\text{if}\ m =2k,\\
\int_\Omega \nabla\Delta^k u \nabla\Delta^k \f \, dx = \int_\Omega g(x,u) \f\, dx, \quad&\text{if}\ m =2k+1.
\end{aligned}
\]
If $g$ satisfies \eqref{eq:1}, then $u\in \bigcap_{1\leq q<\infty}L^q_\mathrm{loc}(\Omega)$.
If $g$ has a polynomial rate of growth:
\[
|g(x,s)|\lesssim 1 + |s|^\frac{N+2m}{N-2m},\quad \text{for all } x\in \Omega\text{ and } s\in \C,
\]
then $u\in W^{2m,q}_\mathrm{loc}(\Omega)\cap C^{2m-1,h}_\mathrm{loc}(\Omega)$ for all $1\leq q< \infty$ and $0<h  < 1$.
\end{corollary}

\section*{Acknowledgements}
The research was partially supported by Narodowe Centrum Nauki Grant No. 2017/26/E/ST1/00817.

\end{document}